\documentclass{article}
\usepackage{amsfonts}
\usepackage{amsmath}
\newtheorem{theorem}{Theorem}[section]

\newtheorem{proposition}[theorem]{Proposition}

\newenvironment{proof}[1][Proof]{\begin{trivlist}
\item[\hskip \labelsep {\bfseries #1}]}{\end{trivlist}}

\newcommand{\qed}{\nobreak \ifvmode \relax \else
      \ifdim\lastskip<1.5em \hskip-\lastskip
      \hskip1.5em plus0em minus0.5em \fi \nobreak
      \vrule height0.75em width0.5em depth0.25em\fi}
\title{The GHS inequality for the Potts Model}
\author{S\'ergio de Carvalho Bezerra\\
\small{affiliation: Universidade Federal de Pernambuco}\\
\small{address: Av. Fernando Sim\~oes Barbosa 316-501}\\
\small{Boa Viagem-Recife-PE-Brazil, CEP 51021-060}\\
\small{email: sergio@de.ufpe.br}}
\begin{document}
\maketitle

\abstract{In this note we are concerned about the generalization of the GHS inequality for the Potts model. We also obtain by a different method the proof of the GHS inequality for the Ising model. We take advantage of a polynomial expansion and we derive that some of the variables of the polynomial can be separated.\textbf{AMS}:60K35, 82B99}
\section{Introduction}
In Statistical Mechanics the Ising model is one of the most important model. One possible extension of the Ising model is the Potts one. An excellent review is \cite{W}. Recent works about the Potts model are \cite{CS},\cite{BR} and \cite{CS2}. The Potts model can be described as follows: Consider $N$ particles and for each particle $i$ we associate a number $\sigma_i\in\Sigma=\{1,2,\ldots,r\}$ called the spin of the particle $i$. $r$ is a positive integer bigger than one. The vector $\sigma=(\sigma_1,\ldots,\sigma_N)\in\Sigma^N$ receives the name of the system configuration. In the nineteenth century people started using one function to describe a physical system: the energy function. In the Potts Model the energy function is:
\begin{align}
E(\sigma)=\sum_{1\leq i<j\leq N} J_{i,j}\delta(\sigma_i,\sigma_j)+\sum_{i=1}^NB_i\delta(1,\sigma_i)
\end{align}
where the numbers $(J_{i,j})_{1\leq i<j\leq N},(B_i)_{1\leq i\leq N}$ will be consider non-negative real numbers and the function $\delta(x,y)=1$ if $x=y$ and zero otherwise. The $J_{i,j}$ represents the exchange interaction between the particles $i$ and $j$. As $J_{i,j}\geq 0$ we have only ferromagnetic interactions. The $B_i$ is the external field in the direction of the particle $i$. The probabilistic model is defined by the triple $(\Sigma^N,\mathcal{P}(\Sigma^N),\mathbb{P})$ where
\begin{align}
\mathbb{P}(\sigma)=\frac{e^{E(\sigma)}}{Z_N}
\end{align}
 the term of normalization $Z_N=\sum_{\sigma\in\Sigma^N}e^{E(\sigma)}$ is named partition function. We obtain the Ising model if we take $r=2$. Let $\underline{J}=(J_{1,2},\ldots,J_{N-1,N})$ and $\underline{B}=(B_1,\ldots,B_N)$, we define the local magnetization as:
 \begin{align}
 m_i(\underline{J},\underline{B})=<\delta(1,\sigma_i)>=\sum_{\sigma\in\Sigma^N}\delta(1,\sigma_i)\mathbb{P}(\sigma).
 \end{align}
   It was proved in \cite{GHS} which is called GHS inequality that
 \begin{theorem}\label{T:first} Let $r$ be equal to two. Then, for all $i,j,k\in\{1,\ldots,N\}$
   \begin{align}
   \frac{\partial m_i(\underline{J},\underline{B})}{\partial B_j\partial B_k}\leq 0.
   \end{align}
 \end{theorem}
  Alternative proofs can be found in \cite{EM} or \cite{L}.  Naturally, we can ask for what does it happen with $r$ bigger than two? This works gives an answer to this question.
In fact, we prove:
\begin{theorem}\label{T:convexity}(GHS inequality for the potts model)Given non-negative real numbers $\underline{J}$ and $\underline{B}$ the local magnetization has the following properties
\begin{itemize}
\item[i)] if $r=2$ then
\begin{align}
 \frac{\partial^2 m_i(\underline{J},\underline{B})}{\partial B_j\partial B_k}&\leq 0
\end{align}
\item[ii)] if $r\geq 3$ then
\begin{align}
 \frac{\partial^2 m_i(\underline{J},\underline{B})}{\partial B_j\partial B_k}&\geq 0.
\end{align}
\end{itemize}
\end{theorem}
Item i) is the theorem~\ref{T:first} which we obtain by a different method.

It was obtained in \cite{E} the proof of the concavity of the magnetization for a model of spins that are real-valued random variables. In the work \cite{EMN} they proved the GHS inequality for families of random variables which
arise in certain ferromagnetic models of statistical mechanics and quantum
field theory.

In the next section, we obtain an expression for the second derivative of the local magnetization. In the following section, we determine that this derivative can be seen as a polynomial. After, we verify that this polynomial can have almost all the variables separated. In the last section, we evaluate some coefficient and prove the GHS inequality for the Potts model.

\section{The second derivative of the local magnetization}
The first step is to evaluate the second derivative of the local magnetization.

\begin{proposition}\label{P:second} Let $r$ be a positive integer bigger than one.  $\underline{J}$ and $\underline{B}$ vectors with entries non-negative. Then, we have
\begin{align}
 &\frac{\partial^2 m_i(\underline{J},\underline{B})}{\partial B_j\partial B_k}=<\delta(1,\sigma_i)\delta(1,\sigma_j)\delta(1,\sigma_k)>-<\delta(1,\sigma_i)\delta(1,\sigma_k)><\delta(1,\sigma_j)>-\nonumber\\
 &-<\delta(1,\sigma_i)\delta(1,\sigma_j)><\delta(1,\sigma_k)>-<\delta(1,\sigma_k)\delta(1,\sigma_j)><\delta(1,\sigma_i)>+\nonumber\\&+2<\delta(1,\sigma_i)><\delta(1,\sigma_k)><\delta(1,\sigma_j)>.\nonumber
\end{align}
\end{proposition}

\begin{proof}if we take the first derivative we have:
\begin{align}
\frac{\partial m_i(\underline{J},\underline{B})}{\partial B_k}=\sum_{\sigma\in\Sigma^N}\delta(1,\sigma_i)\delta(1,\sigma_k)\mathbb{P}(\sigma)-\sum_{\sigma\in\Sigma^N}\sum_{\hat\sigma\in\Sigma^N}\delta(1,\sigma_i)\delta(1,\hat\sigma_k)\mathbb{P}(\sigma)\mathbb{P}(\hat\sigma)\nonumber
\end{align}
Now, we take the second derivative:
\begin{align}
   \frac{\partial^2 m_i(\underline{J},\underline{B})}{\partial B_j\partial B_k}&=\sum_{\sigma\in\Sigma^N}\delta(1,\sigma_i)\delta(1,\sigma_k)\delta(1,\sigma_j)\mathbb{P}(\sigma)-\nonumber\\
&-\sum_{\sigma\in\Sigma^N}\sum_{\hat\sigma\in\Sigma^N}\delta(1,\sigma_i)\delta(1,\sigma_k)\delta(1,\hat\sigma_j)\mathbb{P}(\sigma)\mathbb{P}(\hat\sigma)\nonumber\\
&-\sum_{\sigma\in\Sigma^N}\sum_{\hat\sigma\in\Sigma^N}\delta(1,\sigma_i)\delta(1,\hat\sigma_k)\delta(1,\sigma_j)\mathbb{P}(\sigma)\mathbb{P}(\hat\sigma)\nonumber\\
&-\sum_{\sigma\in\Sigma^N}\sum_{\hat\sigma\in\Sigma^N}\delta(1,\sigma_i)\delta(1,\hat\sigma_k)\delta(1,\hat\sigma_j)\mathbb{P}(\sigma)\mathbb{P}(\hat\sigma)\nonumber\\
&+2\sum_{\sigma\in\Sigma^N}\sum_{\hat\sigma\in\Sigma^N}\sum_{\check\sigma\in\Sigma^N}\delta(1,\sigma_i)\delta(1,\hat\sigma_k)\delta(1,\check\sigma_j)\mathbb{P}(\sigma)\mathbb{P}\hat\sigma)\mathbb{P}(\check\sigma)\nonumber
\end{align}
\qed
\end{proof}
We pick up the idea introduced in \cite{GHS} and we consider a ghost spin $\sigma_0$. Consequently, we consider the energy as
\begin{align}
E(\sigma)=\sum_{1\leq i<j\leq N} J_{i,j}\delta(\sigma_i,\sigma_j)+\sum_{i=1}^NB_i\delta(\sigma_0,\sigma_i)=\sum_{0\leq i<j\leq N} J_{i,j}\delta(\sigma_i,\sigma_j)\nonumber
\end{align}
where $J_{0,i}=B_i$ and if we set up $i=1,j=2,k=3$ then the second derivative becomes:
\begin{align}
& \frac{\partial^2 m_1(\underline{J},\underline{B})}{\partial B_2\partial B_3}=<\delta(\sigma_0,\sigma_1)\delta(\sigma_0,\sigma_2)\delta(\sigma_0,\sigma_3)>-\nonumber\\&-<\delta(\sigma_0,\sigma_1)\delta(\sigma_0,\sigma_3)><\delta(\sigma_0,\sigma_2)>-
 <\delta(\sigma_0,\sigma_1)\delta(\sigma_0,\sigma_2)><\delta(\sigma_0,\sigma_3)>-\nonumber\\&-<\delta(\sigma_0,\sigma_3)\delta(\sigma_0,\sigma_2)><\delta(\sigma_0,\sigma_1)>+\nonumber\\&+2<\delta(\sigma_0,\sigma_1)><\delta(\sigma_0,\sigma_2)><\delta(\sigma_0,\sigma_3)>.\nonumber
\end{align}
We will study the signal of this last expression. By one remark in \cite{GHS} this signal implies the signal of the second derivative of the Proposition~\ref{P:second}.
\section{The recursive property of the second derivative signal}
We define
 $$H=H(\sigma)=\exp{(\displaystyle\sum_{0\leq i<j\leq N}J_{i,j}\delta(\sigma_i,\sigma_j))}.$$
 Thus, we pay attention for the following term:

\begin{align}\label{E:def_I}
&I(\underline{J},\underline{B})\colon= Z_N^3\frac{\partial^2 m_1(\underline{J},\underline{B})}{\partial B_2\partial B_3}=\sum_{\begin{array}{c}\sigma\end{array}}H\sum_{\begin{array}{c}\sigma\end{array}}H\sum_{\begin{array}{c}\sigma\\
\sigma_0=\sigma_1=\sigma_2=\sigma_3\end{array}}H-\nonumber\\&\sum_{\begin{array}{c}\sigma\end{array}}H\sum_{\begin{array}{c}\sigma\\\sigma_0=\sigma_1=\sigma_2\end{array}}H\sum_{\begin{array}{c}\sigma\\
\sigma_0=\sigma_3\end{array}}H-\nonumber\\&-\sum_{\begin{array}{c}\sigma\end{array}}H\sum_{\begin{array}{c}\sigma\\\sigma_0=\sigma_1=\sigma_3\end{array}}H\sum_{\begin{array}{c}\sigma\\
\sigma_0=\sigma_2\end{array}}H\nonumber\\&-\sum_{\begin{array}{c}\sigma\end{array}}H\sum_{\begin{array}{c}\sigma\\\sigma_0=\sigma_2=\sigma_3\end{array}}H\sum_{\begin{array}{c}\sigma\\
\sigma_0=\sigma_1\end{array}}H+\nonumber\\&+2\sum_{\begin{array}{c}\sigma\\\sigma_0=\sigma_1\end{array}}H\sum_{\begin{array}{c}\sigma\\\sigma_0=\sigma_2\end{array}}H\sum_{\begin{array}{c}\sigma\\
\sigma_0=\sigma_3\end{array}}H.
\end{align}

Given a positive integer $s$ let $\mathcal{A}_{s}$ be the set of matrix $A=(a_{i,j})$ with index $1\leq i\leq \binom{N+1}{2}$ and $ 1\leq j\leq 3$ where $a_{i,j}\in\{0,1\}$ if the index $i$ is bigger or equal to $\binom{N+1}{2}-s+1$ and less or equal than $\binom {N+1}{2}$.Further, $a_{i,j}=0$ otherwise. We consider that the pairs of particles are in a specific order $\mathcal{O}=\{(0,1),(0,2),\ldots,(N-1,N)\}=\{o_1,\ldots o_{\binom{N+1}{2}}\}$. The row $p$ in $A$ is associated to the pair of particles $o_p$.Let $H_{s}$ be $H_{s}(\sigma)=\exp{(\displaystyle\sum_{(i,j)\in \mbox{ the first }\binom{N+1}{2}-s\mbox{ terms in } \mathcal{O}}J_{i,j}\delta(\sigma_i,\sigma_j))}$. For each matrix $A\in \mathcal{A}_s$ we associate the term
$I_A=I_{1,A}+I_{2,A}-I_{3,A}-I_{4,A}+I_{5,A}$ with
\begin{align}\label{E:def_I12345}
&I_{1,A}=\big(\sum_{\begin{array}{c}\sigma\\\sigma_i=\sigma_j\mbox{ if }\\o_r=(i,j),a_{r,1}=1\end{array}}H_s\sum_{\begin{array}{c}\sigma\\\sigma_i=\sigma_j\mbox{ if }\\o_r=(i,j),a_{r,2}=1\end{array}}H_s\big)\times\nonumber\\&\times\sum_{\begin{array}{c}\sigma\\
\sigma_0=\sigma_1=\sigma_2=\sigma_3\\\sigma_i=\sigma_j\mbox{ if }\\o_r=(i,j),a_{r,3}=1\end{array}}H_s\nonumber
\end{align}
\begin{align}
I_{2,A}&=\big(\sum_{\begin{array}{c}\sigma\\\sigma_i=\sigma_j\mbox{ if }\\o_r=(i,j),a_{r,1}=1\end{array}}H_s\sum_{\begin{array}{c}\sigma\\\sigma_0=\sigma_1=\sigma_3\\\sigma_i=\sigma_j\mbox{ if }\\o_r=(i,j),a_{r,2}=1\end{array}}H_s\big)\times\nonumber\\&\times
\sum_{\begin{array}{c}\sigma\\
\sigma_0=\sigma_2\\\sigma_i=\sigma_j\mbox{ if }\\o_r=(i,j),a_{r,3}=1\end{array}}H_s\nonumber
\end{align}
\begin{align}
I_{3,A}&=\big(\sum_{\begin{array}{c}\sigma\\\sigma_i=\sigma_j\mbox{ if }\\o_r=(i,j),a_{r,1}=1\end{array}}H_s\sum_{\begin{array}{c}\sigma\\\sigma_0=\sigma_1=\sigma_3\\\sigma_i=\sigma_j\mbox{ if }\\o_r=(i,j),a_{r,2}=1\end{array}}H_s\big)\times\nonumber\\&\times\sum_{\begin{array}{c}\sigma\\
\sigma_0=\sigma_2\\\sigma_i=\sigma_j\mbox{ if }\\o_r=(i,j),a_{r,3}=1\end{array}}H_s\nonumber
\end{align}
\begin{align}
I_{4,A}&=\big(\sum_{\begin{array}{c}\sigma\\\sigma_i=\sigma_j\mbox{ if }\\o_r=(i,j),a_{r,1}=1\end{array}}H_s\sum_{\begin{array}{c}\sigma\\\sigma_0=\sigma_2=\sigma_3\\\sigma_i=\sigma_j\mbox{ if }\\o_r=(i,j),a_{r,2}=1\end{array}}H_s\big)\times\nonumber\\
&\times\sum_{\begin{array}{c}\sigma\\
\sigma_0=\sigma_1\\\sigma_i=\sigma_j\mbox{ if }\\o_r=(i,j),a_{r,3}=1\end{array}}H_s\nonumber
\end{align}
\begin{align}
I_{5,A}&=2\big(\sum_{\begin{array}{c}\sigma\\\sigma_0=\sigma_1\\\sigma_i=\sigma_j\mbox{ if }\\o_r=(i,j),a_{r,1}=1\end{array}}H_s\sum_{\begin{array}{c}\sigma\\\sigma_0=\sigma_2\\\sigma_i=\sigma_j\mbox{ if }\\o_r=(i,j),a_{r,2}=1\end{array}}H_s\big)\times\nonumber\\&\times\sum_{\begin{array}{c}\sigma\\
\sigma_0=\sigma_3\\\sigma_i=\sigma_j\mbox{ if }\\o_r=(i,j),a_{r,3}=1\end{array}}H_s.\nonumber
\end{align}
It means that we have many constraints of the type $\sigma_i=\sigma_j$ which are defined by the element $A$.
As For each row of $A$ we have a pair $o_p=(i,j)$ we consider its weight defined by $n(p)\colon=a(p,1)+a(p,2)+a(p,3)$ and also we designate by $J_p\colon=J_{i,j}$. Now, we can announce the following result which can be interpreted as $I(\underline{J},\underline{B})$ is a polynomial in the variables $X_p\colon=(e^{J_p}-1)$ with $p=1,\ldots,\binom{N+1}{2}$.
\begin{proposition}Let $\underline{J}$ and $\underline{B}$ be non-negative real numbers and $s$ be a positive integer that belongs to $\{1,\ldots,\binom{N+1}{2}\}$. Then
\begin{align}
I(\underline{J},\underline{B})=\sum_{A\in \mathcal{A}_s}I_AX_{\binom{N+1}{2}-s+1}^{n(\binom{N+1}{2}-s+1)}\ldots X_{\binom{N+1}{2}}^{n(\binom{N+1}{2})}
\end{align}
\end{proposition}
A particular case in which we are interested is when $s=\binom{N+1}{2}$ we have \begin{align}\label{E:simplificada}
I(\underline{J},\underline{B})=\sum_{A\in A_{\binom{N+1}{2}}}I_AX_{1}^{n(1)}\ldots X_{\binom{N+1}{2}}^{n(\binom{N+1}{2})}
\end{align}
we remark that in this case $H_{\binom{N+1}{2}}=1$.
\begin{proof}
The proof is by induction on $s$. Let $s$ be equal to one. We observe that
\begin{align}\label{E:eq1}
\sum_{\begin{array}{c}\sigma\end{array}}H&=\sum_{\begin{array}{c}\sigma\\\sigma_{N-1}=\sigma_{N}\end{array}}H_1e^{J_{\binom{N+1}{2}}}+\sum_{\begin{array}{c}\sigma\\\sigma_{N-1}\neq\sigma_N\end{array}}H_1\nonumber\\&=
\sum_{\begin{array}{c}\sigma\\\sigma_{N-1}=\sigma_{N}\end{array}}H_1e^{J_{\binom{N+1}{2}}}+\sum_{\begin{array}{c}\sigma\end{array}}H_1-\sum_{\begin{array}{c}\sigma\\\sigma_{N-1}=\sigma_N\end{array}}H_1\nonumber\\
&=\sum_{\begin{array}{c}\sigma\\\sigma_{N-1}=\sigma_{N}\end{array}}H_1(e^{J_{\binom{N+1}{2}}}-1)+\sum_{\begin{array}{c}\sigma\end{array}}H_1.
\end{align}
In the second step we have just used the fact that $$\sum_{\begin{array}{c}\sigma\end{array}}H_1=\sum_{\begin{array}{c}\sigma\\\sigma_{N-1}=\sigma_N\end{array}}H_1+\sum_{\begin{array}{c}\sigma\\\sigma_{N-1}\neq\sigma_N\end{array}}H_1.$$
We use the notation $[i_1,\ldots,i_n]$ if $\sigma_{i_1}=\ldots=\sigma_{i_n}$.
Then, If we replace expression ~(\ref{E:eq1}) for each one of the elements in the equation~(\ref{E:def_I}) we obtain:
\begin{align}
I(\underline{J},\underline{B})=&\big(\sum_{\begin{array}{c}\sigma\\\, [N-1,N]\end{array}}H_1(e^{J_{\binom{N+1}{2}}}-1)+\sum_{\begin{array}{c}\sigma\end{array}}H_1\big)\times\nonumber\\&\times\big(\sum_{\begin{array}{c}\sigma\\\,[N-1,N]\end{array}}H_1(e^{J_{\binom{N+1}{2}}}-1)+\sum_{\begin{array}{c}\sigma\end{array}}H_1\big)\times\nonumber
\end{align}
\begin{align}
\hspace{3cm}
&\times\big(\sum_{\begin{array}{c}\sigma\\\,[0,1,2,3]\\\,[N-1,N]\end{array}}H_1(e^{J_{\binom{N+1}{2}}}-1)+\sum_{\begin{array}{c}\sigma\\\,[0,1,2,3]\end{array}}H_1\big)-\nonumber
\end{align}
\begin{align}
\hspace{3cm}
&\big(\sum_{\begin{array}{c}\sigma\\\,[N-1,N]\end{array}}H_1(e^{J_{\binom{N+1}{2}}}-1)+\sum_{\begin{array}{c}\sigma\end{array}}H_1\big)\times\nonumber
\end{align}
\begin{align}
\hspace{3cm}
&\times\big(\sum_{\begin{array}{c}\sigma\\\,[0,1,2]\\\,[N-1,N]\end{array}}H_1(e^{J_{\binom{N+1}{2}}}-1)+\sum_{\begin{array}{c}\sigma\\\,[0,1,2]\end{array}}H_1\big)\times\nonumber
\end{align}
\begin{align}
\hspace{3cm}
&\times\big(\sum_{\begin{array}{c}\sigma\\\,[0,3]\\\,[N-1,N]\end{array}}H_1(e^{J_{\binom{N+1}{2}}}-1)+\sum_{\begin{array}{c}\sigma\\\,[0,3]\end{array}}H_1\big)-\nonumber
\end{align}
\begin{align}
\hspace{3cm}
&\big(\sum_{\begin{array}{c}\sigma\\\,[N-1,N]\end{array}}H_1(e^{J_{\binom{N+1}{2}}}-1)+\sum_{\begin{array}{c}\sigma\end{array}}H_1\big)\times\nonumber
\end{align}
\begin{align}
\hspace{3cm}
&\times\big(\sum_{\begin{array}{c}\sigma\\\,[0,1,3]\\\,[N-1,N]\end{array}}H_1(e^{J_{\binom{N+1}{2}}}-1)+\sum_{\begin{array}{c}\sigma\\\,[0,1,3]\end{array}}H_1\big)\times\nonumber
\end{align}
\begin{align}
\hspace{3cm}
&\times\big(\sum_{\begin{array}{c}\sigma\\\,[0,2]\\\,[N-1,N]\end{array}}H_1(e^{J_{\binom{N+1}{2}}}-1)+\sum_{\begin{array}{c}\sigma\\\,[0,2]\end{array}}H_1\big)-\nonumber
\end{align}
\begin{align}
\hspace{3cm}
&\big(\sum_{\begin{array}{c}\sigma\\\,[N-1,N]\end{array}}H_1(e^{J_{\binom{N+1}{2}}}-1)+\sum_{\begin{array}{c}\sigma\end{array}}H_1\big)\times\nonumber
\end{align}
\begin{align}
\hspace{3cm}
&\times\big(\sum_{\begin{array}{c}\sigma\\\,[0,2,3]\\\,[N-1,N]\end{array}}H_1(e^{J_{\binom{N+1}{2}}}-1)+\sum_{\begin{array}{c}\sigma\\\,[0,2,3]\end{array}}H_1\big)\times\nonumber
\end{align}
\begin{align}
\hspace{3cm}
&\times\big(\sum_{\begin{array}{c}\sigma\\\,[0,1]\\\,[N-1,N]\end{array}}H_1(e^{J_{\binom{N+1}{2}}}-1)+\sum_{\begin{array}{c}\sigma\\\,[0,1]\end{array}}H_1\big)+\nonumber
\end{align}
\begin{align}
\hspace{3cm}
&+2\big(\sum_{\begin{array}{c}\sigma\\\,[0,1]\\\,[N-1,N]\end{array}}H_1(e^{J_{\binom{N+1}{2}}}-1)+\sum_{\begin{array}{c}\sigma\\,[0,1]\end{array}}H_1\big)\times\nonumber
\end{align}
\begin{align}
\hspace{3cm}
&\times\big(\sum_{\begin{array}{c}\sigma\\\,[0,2]\\\,[N-1,N]\end{array}}H_1(e^{J_{\binom{N+1}{2}}}-1)+\sum_{\begin{array}{c}\sigma\\\,[0,2]\end{array}}H_1\big)\times\nonumber
\end{align}
\begin{align}
\hspace{3cm}
&\times\big(\sum_{\begin{array}{c}\sigma\\\,[0,3]\\\,[N-1,N]\end{array}}H_1(e^{J_{\binom{N+1}{2}}}-1)+\sum_{\begin{array}{c}\sigma\\\,[0,3]\end{array}}H_1\big).\nonumber
\end{align}
Thus, the result follows. Now, we consider that we proved the formula for a positive integer $s$. That means,
\begin{align}
I(\underline{J},\underline{B})=\sum_{A\in \mathcal{A}_s}I_AX_{\binom{N+1}{2}-s+1}^{n(\binom{N+1}{2}-s+1)}\ldots X_{\binom{N+1}{2}}^{n(\binom{N+1}{2})}.\label{E:def_Is}
\end{align}
Let $p=\binom{N}{2}-s-1$ with $o_{p}=(i,j)$. Then,
\begin{align}
H_s=H_{s+1}e^{J_p\delta(\sigma_i,\sigma_j)}
\end{align}
we have that
\begin{align}
\sum_{\begin{array}{c}\sigma\\\mbox{condition on the }\sigma_k\mbox{'s}\end{array}}H_s&=\sum_{\begin{array}{c}\sigma\\\mbox{condition on the }\sigma_k\mbox{'s}\\\sigma_i=\sigma_j\end{array}}H_{s+1}e^{J_p}+\nonumber\\&+\sum_{\begin{array}{c}\sigma\\\mbox{condition on the }\sigma_k\mbox{'s}\\\sigma_i\neq\sigma_j\end{array}}H_{s+1}\nonumber\\
&=\sum_{\begin{array}{c}\sigma\\\mbox{condition on the }\sigma_k\mbox{'s}\\\sigma_i=\sigma_j\end{array}}H_{s+1}(e^{J_p}-1)+\nonumber\\&+\sum_{\begin{array}{c}\sigma\\\mbox{condition on the }\sigma_k\mbox{'s}\end{array}}H_{s+1}.\label{E:eq2}
\end{align}
In the last step, we used the fact that
\begin{align}
\sum_{\begin{array}{c}\sigma\\\mbox{condition on the }\sigma_k\mbox{'s}\\\sigma_i\neq\sigma_j\end{array}}H_{s+1}&=\sum_{\begin{array}{c}\sigma\\\mbox{condition on the }\sigma_k\mbox{'s}\end{array}}H_{s+1}-\nonumber\\&-\sum_{\begin{array}{c}\sigma\\\mbox{condition on the }\sigma_k\mbox{'s}\\\sigma_i=\sigma_j\end{array}}H_{s+1}.\nonumber
\end{align}
If we use expression~(\ref{E:eq2}) in the defintion of the term $I_A$ and we replace the result in the formula~(\ref{E:def_Is}) we prove the proposition.
\qed
\end{proof}
\section{The separation formula}

At this moment, we obtain the separation formula of the variable $X_i$'s. Let $O_1=O_2\cup O_3$ where $O_2=\{(1,2),(1,3),(2,3)\}$ and the set $O_3$ be equal to $\{(0,1),(0,2),(0,3)\}$. The pair of particles of $O_2$ will take a particular importance, we denote them by $p_1=(1,2),p_2=(1,3),p_3=(2,3)$.
\begin{proposition}(Separation formula)
\begin{align}
I(\underline{J},\underline{B})&=\prod_{\begin{array}{c}p=1\\o_p\notin O_1 \end{array}}^{\binom{N+1}{2}}(1+r^{-1}X_p)^3\times\nonumber\\&\times\prod_{\begin{array}{c}p=1\\o_p\in O_3 \end{array}}^{\binom{N+1}{2}}(1+(1+2r^{-1})X_p+(2r^{-1}+r^{-2})X_p^2+r^{-2}X_p^3)\times\nonumber\\&\times\sum_{\begin{array}{c}A\in \mathcal{A}_{\binom{N+1}{2}}\\\mbox{if }o_p\notin O_2\\a(p,1)=a(p,2)=a(p,3)=0\end{array}}I_AX_{p_1}^{n(p_1)}X_{p_2}^{n(p_2)}X_{p_3}^{n(p_3)}.\nonumber
\end{align}
\end{proposition}
\begin{proof}
We start by the expression~(\ref{E:simplificada}). Given $A$ we define $\hat A$ as the matrix with entries equal to the entries of $A$ in all rows except in row $p$. At row $p$ we have $\hat a(p,1)=\hat a(p,2)=\hat a(p,3)=0$.

 We consider the case that $p\notin O_1$. Let $n(p)$ be the weight of the row $p$ of $A$. We observe the following relations, if $n(p)=0$ then $I_A=I_{\hat A}$. In the case that $n(p)=1$ then $I_A=r^{-1}I_{\hat A}$. For $n(p)=2$ we have $I_A=r^{-2}I_{\hat A}$. Finally, when $n(p)=3$ we get $I_A=r^{-3}I_{\hat A}$. We do the calculations for $n(p)=1$. let $p$ be the pair $(i,j)$ then $n(p)=1$ means that appears the condition $\sigma_i=\sigma_j$ in one term of $I_{1,A},I_{2,A},I_{3,A},I_{4,A}$ and $I_{5,A}$ defined in (\ref{E:def_I12345}). Furthermore, this is the only additional constraint if we compare the conditions implied by $A$ and $\hat A$. These constraints define a partition of the elements $\{0,1,\ldots,N\}$ where two elements $\ell,\hat\ell$ belong to the same subset if $\sigma_{\ell}=\sigma_{\hat\ell}$. When we sum up $$\displaystyle\sum_{\begin{array}{c}\sigma\\\mbox{ some constraints of the type}\\\sigma_{\ell}=\sigma_{\hat\ell}\end{array}}1$$ we obtain $r^{S}$ where $S$ is the number of subsets of the partition defined by the constraints. In the case that we are considering, we have one additional constraint if we compare $A$ with $\hat A$. Thus, the partition implied by the constraints of $A$ has the number of subsets of the partition implied by $\hat A$ less one. Further, as $(i,j)\notin O_1$ the constraint $\sigma_i=\sigma_j$ is a new constraint. It is not an existent condition as $\sigma_0=\sigma_1=\sigma_2=\sigma_3$ in $I_{1,A}$. Thus, $I_A=r^{-1}I_{\hat A}$.
  We observe that there exists eight elements belong to $\mathcal{A}_{\binom{N+1}{2}}$ that we associate the same $\hat A$. They differ one from each other by the possible values of $a(p,1),a(p,2)$ and $a(p,3)$. We have one element with $n(p)=0$, three elements with $n(p)=1$, three elements with $n(p)=2$ and one element with $n(p)=3$. Hence, the formula
  \begin{align}
  I(\underline{J},\underline{B})&=\sum_{\begin{array}{c}A\in\mathcal{A}_{\binom{N+1}{2}}\\n(p)=0\end{array}}I_{A}X_{1}^{n(1)}\ldots X_{p}^{0}\ldots X_{\binom{N+1}{2}}^{n(\binom{N+1}{2})}+\nonumber\\&+\sum_{\begin{array}{c}A\in\mathcal{A}_{\binom{N+1}{2}}\\n(p)=1\end{array}}I_{A}X_{1}^{n(1)}\ldots X_{p}^{1}\ldots X_{\binom{N+1}{2}}^{n(\binom{N+1}{2})}+\nonumber\\&+\sum_{\begin{array}{c}A\in\mathcal{A}_{\binom{N+1}{2}}\\n(p)=2\end{array}}I_{A}X_{1}^{n(1)}\ldots X_{p}^{2}\ldots X_{\binom{N+1}{2}}^{n(\binom{N+1}{2})}+\nonumber\\&+\sum_{\begin{array}{c}A\in\mathcal{A}_{\binom{N+1}{2}}\\n(p)=3\end{array}}I_{A}X_{1}^{n(1)}\ldots X_{p}^{3}\ldots X_{\binom{N+1}{2}}^{n(\binom{N+1}{2})}\nonumber
  \end{align}
  can be simplified
  \begin{align}
    I(\underline{J},\underline{B})&=
  \sum_{\begin{array}{c}A\in\mathcal{A}_{\binom{N+1}{2}}\\n(p)=0\end{array}}I_{A}X_{1}^{n(1)}\ldots \hat X_{p}\ldots X_{\binom{N+1}{2}}^{n(\binom{N+1}{2})}+\nonumber\\&+3r^{-1}\sum_{\begin{array}{c}A\in\mathcal{A}_{\binom{N+1}{2}}\\n(p)=0\end{array}}I_{A}X_{1}^{n(1)}\ldots \hat X_{p}\ldots X_{\binom{N+1}{2}}^{n(\binom{N+1}{2})}X_{p}+\nonumber\\&+3r^{-2}\sum_{\begin{array}{c}A\in\mathcal{A}_{\binom{N+1}{2}}\\n(p)=0\end{array}}I_{A}X_{1}^{n(1)}\ldots \hat X_{p}\ldots X_{\binom{N+1}{2}}^{n(\binom{N+1}{2})}X_{p}^{2}+\nonumber\\&+\sum_{\begin{array}{c}A\in\mathcal{A}_{\binom{N+1}{2}}\\n(p)=0\end{array}}I_{A}X_{1}^{n(1)}\ldots \hat X_{p}\ldots X_{\binom{N+1}{2}}^{n(\binom{N+1}{2})}X_{p}^{3}\nonumber
  \end{align}
  where $\hat X$ means that this term does not appear. With more simplifications, we obtain
  \begin{align}
    I(\underline{J},\underline{B})&=
    \sum_{\begin{array}{c}A\in\mathcal{A}_{\binom{N+1}{2}}\\n(p)=0\end{array}}I_{A}X_{1}^{n(1)}\ldots \hat X_{p}\ldots X_{\binom{N+1}{2}}^{n(\binom{N+1}{2})}(1+3r^{-1}X_p+\nonumber\\&+3r^{-2}X_p^2+r^{-3}X_p^3)\nonumber\\&
    =  \sum_{\begin{array}{c}A\in\mathcal{A}_{\binom{N+1}{2}}\\n(p)=0\end{array}}I_{A}X_{1}^{n(1)}\ldots \hat X_{p}\ldots X_{\binom{N+1}{2}}^{n(\binom{N+1}{2})}(1+r^{-1}X_p)^3.\nonumber
  \end{align}
We can repeat the procedure for all $p\notin O_3$ which produces
\begin{align}
    I(\underline{J},\underline{B})&=\prod_{\begin{array}{c}p=1\\o_p\notin O_1 \end{array}}^{\binom{N+1}{2}}(1+r^{-1}X_p)^3\times\nonumber\\&\times\sum_{\begin{array}{c}A\in \mathcal{A}_{\binom{N+1}{2}}\\\mbox{if }o_p\notin O_3\\a(p,1)=a(p,2)=a(p,3)=0\end{array}}I_A\prod_{\begin{array}{c}p=1\\o_p\in O_3\end{array}}^{\binom{N+1}{2}}X_{p}^{n(p)}.\label{E:parcial}
\end{align}
Now, we pay attention to the terms that belong to $O_1$. Let $p_4$ be $(0,1)$. For a matrix $A\in\mathcal{A}_{\binom{N+1}{2}}$ with entries $a(p,1)=a(p,2)=a(p,3)=0$ for all $p\notin O_3$ we associate the matrix $\hat A$ with the same entries except for row $p_4$ we have $\hat a(p,1)=\hat a(p,2)=\hat a(p,3)=0$. We can show that if $n(p_4)=0$ then we have $I_A=I_{\hat A}$. Also, we obtain
\begin{align}
&\sum_{\begin{array}{c}A\in\mathcal{A}_{\binom{N+1}{2}}\\n(p_4)=1\\n(p)=0\forall p\in O_3\end{array}}I_AX_{p_4}\prod_{\begin{array}{c}p=1\\o_p\in O_3\\p\neq p_4\end{array}}^{\binom{N+1}{2}}X_{p}^{n(p)} \nonumber\\&=(1+2r^{-1})X_{p_4}\sum_{\begin{array}{c}A\in\mathcal{A}_{\binom{N+1}{2}}\\n(p_4)=0\\n(p)=0\forall p\in O_3\end{array}}I_A\prod_{\begin{array}{c}p=1\\o_p\in O_3\\p\neq p_4\end{array}}^{\binom{N+1}{2}}X_{p}^{n(p)}.\label{E:eq3}
\end{align}
For $n(p_4)=2$ we have
\begin{align}
&\sum_{\begin{array}{c}A\in\mathcal{A}_{\binom{N+1}{2}}\\n(p_4)=2\\n(p)=0\forall p\in O_3\end{array}}I_AX_{p_4}^2\prod_{\begin{array}{c}p=1\\o_p\in O_3\\p\neq p_4\end{array}}^{\binom{N+1}{2}}X_{p}^{n(p)} \nonumber\\&=(2r^{-1}+r^{-2})X_{p_4}^2\sum_{\begin{array}{c}A\in\mathcal{A}_{\binom{N+1}{2}}\\n(p_4)=0\\n(p)=0\forall p\in O_3\end{array}}I_A\prod_{\begin{array}{c}p=1\\o_p\in O_3\\p\neq p_4\end{array}}^{\binom{N+1}{2}}X_{p}^{n(p)}.\label{E:eq4}
\end{align}
And if $n(p_4)=3$ we get $I_{A}=r^{-2}I_{\hat A}$.
Again, we have eight matrix $A$ which are associated for the same $\hat A$. We consider the three matrix with $n(p_4)=1$, i.e., the matrix $A_1$ with $a(p_4,1)=1$ the matrix $A_2$ with $a(p_4,2)=1$ and $A_3$ with $a(p_4,3)=1$. We have that $I_{1,A_1}=r^{-1}I_{1,\hat A}$,$I_{2,A_1}=r^{-1}I_{2,\hat A}$,$I_{3,A_1}=r^{-1}I_{3,\hat A}$,$I_{4,A_1}=r^{-1}I_{4,\hat A}$ and $I_{5,A_1}=I_{5,\hat A}$. In the last term does not appear the factor $r^{-1}$ because the condition $\sigma_0=\sigma_1$ appears independently of the matrix $A_1$. We also obtain that $I_{1,A_2}=r^{-1}I_{1,\hat A}$,$I_{2,A_2}=I_{2,\hat A}$,$I_{3,A_2}=I_{3,\hat A}$,$I_{4,A_2}=r^{-1}I_{4,\hat A}$ and $I_{5,A_2}=r^{-1}I_{5,\hat A}$. There are two terms that do not appear $r^{-1}$ it happens because again the condition $\sigma_0=\sigma_1$ occurs independently of the matrix $A_2$. For $A_3$ we have $I_{1,A_3}=I_{1,\hat A}$,$I_{2,A_3}=r^{-1}I_{2,\hat A}$,$I_{3,A_3}=r^{-1}I_{3,\hat A}$,$I_{4,A_3}=I_{4,\hat A}$ and $I_{5,A_3}=r^{-1}I_{5,\hat A}$. Again the terms that do not appear $r^{-1}$ are consequence of the constraint $\sigma_0=\sigma_1$ happens independently of the matrix $A_3$.  Besides, equations~(\ref{E:eq3}) is a consequence of these facts. Now we look at the three matrices with $n(p)=2$ let $A_4$ be the matrix with $a(p_4,1)=a(p_4,2)=1$ $A_5$ be the matrix with $a(p_4,1)=a(p_4,3)=1$ and $A_5$ the one with $a(p_4,2)=a(p_4,3)=1$. We obtain that $I_{1,A_4}=r^{-2}I_{1,\hat A}$,$I_{2,A_4}=r^{-1}I_{2,\hat A}$,$I_{3,A_4}=r^{-1}I_{3,\hat A}$,$I_{4,A_4}=r^{-2}I_{4,\hat A}$ and $I_{5,A_4}=r^{-1}I_{5,\hat A}$. The terms $r^{-1}$ appear when one of the conditions $\sigma_0=\sigma_1$ already occurs independently of $A_4$ and the term $r^{-2}$ appears when both conditions $\sigma_0=\sigma_1$ are new. For the matrix $A_5$ we obtain that $I_{1,A_5}=r^{-1}I_{1,\hat A}$,$I_{2,A_5}=r^{-2}I_{2,\hat A}$,$I_{3,A_5}=r^{-2}I_{3,\hat A}$,$I_{4,A_5}=r^{-1}I_{4,\hat A}$ and $I_{5,A_5}=r^{-1}I_{5,\hat A}$ with the same reasoning. And if we see the matrix $A_6$ we get $I_{1,A_6}=r^{-1}I_{1,\hat A}$,$I_{2,A_6}=r^{-1}I_{2,\hat A}$,$I_{3,A_6}=r^{-1}I_{3,\hat A}$,$I_{4,A_6}=r^{-1}I_{4,\hat A}$ and $I_{5,A_6}=r^{-2}I_{5,\hat A}$. Thus, we prove equation~(\ref{E:eq4}). As a matter of facts we obtain from equation~(\ref{E:parcial}) that
\begin{align}
    I(\underline{J},\underline{B})&=\prod_{\begin{array}{c}p=1\\o_p\notin O_1 \end{array}}^{\binom{N+1}{2}}(1+r^{-1}X_p)^3\times\nonumber
    \end{align}
\begin{align}
\hspace{3cm}
    &\times\left[\sum_{\begin{array}{c}A\in \mathcal{A}_{\binom{N+1}{2}}\\n(p_4)=0\mbox{ if }o_p\notin O_3\\n(p)=0\end{array}}I_A\prod_{\begin{array}{c}p=1\\o_p\in O_3\\p\neq p_4\end{array}}^{\binom{N+1}{2}}X_{p}^{n(p)}\right.+\nonumber
\end{align}
\begin{align}
\hspace{3cm}
    &+\sum_{\begin{array}{c}A\in \mathcal{A}_{\binom{N+1}{2}}\\n(p_4)=1\mbox{ if }o_p\notin O_3\\n(p)=0\end{array}}I_AX_{p_4}\prod_{\begin{array}{c}p=1\\o_p\in O_3\\p\neq p_4\end{array}}^{\binom{N+1}{2}}X_{p}^{n(p)}+\nonumber
    \end{align}
\begin{align}
\hspace{3cm}
    &+\sum_{\begin{array}{c}A\in \mathcal{A}_{\binom{N+1}{2}}\\n(p_4)=2\mbox{ if }o_p\notin O_3\\n(p)=0\end{array}}I_AX_{p_4}^2\prod_{\begin{array}{c}p=1\\o_p\in O_3\\p\neq p_4\end{array}}^{\binom{N+1}{2}}X_{p}^{n(p)}+\nonumber
    \end{align}
\begin{align}
\hspace{3cm}
    &+\left.\sum_{\begin{array}{c}A\in \mathcal{A}_{\binom{N+1}{2}}\\n(p_4)=3\mbox{ if }o_p\notin O_3\\n(p)=0\end{array}}I_AX_{p_4}^3\prod_{\begin{array}{c}p=1\\o_p\in O_3\\p\neq p_4\end{array}}^{\binom{N+1}{2}}X_{p}^{n(p)}\right]\nonumber
\end{align}
By the remarks above and equation~(\ref{E:eq3}) and~(\ref{E:eq4}) we conclude that
\begin{align}
    I(\underline{J},\underline{B})&=\prod_{\begin{array}{c}p=1\\o_p\notin O_1 \end{array}}^{\binom{N+1}{2}}(1+r^{-1}X_p)^3\times\nonumber\\&\times
    (1+(1+2r^{-1})X_{p_4}+(2r^{-1}+r^{-2})X_{p_4}^2+r^{-2}X_{p_4}^3)\times
    \nonumber\\&\times\sum_{\begin{array}{c}A\in \mathcal{A}_{\binom{N+1}{2}}\\n(p_4)=0\mbox{ if }o_p\notin O_3\\n(p)=0\end{array}}I_A\prod_{\begin{array}{c}p=1\\p\neq p_4\,o_p\in O_3\end{array}}^{\binom{N+1}{2}}X_{p}^{n(p)}.
\end{align}
We can repeat this last procedure for the other two elements of $O_3$ which ends the proof.
\qed
\end{proof}
\section{Some coefficients and the main theorem}
Now, We will evaluate the constants $I_A$ for $A\in\mathcal{A}_{\binom{N+1}{2}}$ with $n(p)=0$ if $o_p\notin O_2$. These matrices can have entries different from zero only in three rows, we will represent these matrices by matrices $3\times 3$ where the first row is associated to the pair $(1,2)$ the second one to the pair $(1,3)$ and the third one to $(2,3)$. We also define the coefficients $$\alpha(x,y,z)=\sum_{\begin{array}{c}A\in\mathcal{A}_{\binom{N+1}{2}}\\n(p_1)=x,n(p_2)=y\\n(p_3)=z,\mbox{ for the other } p\\n(p)=0\end{array}}I_A$$

Using the fact that if $\sigma_1=\sigma_2$ and $\sigma_1=\sigma_3$ implies that $\sigma_2=\sigma_3$ and the symmetries, we get
\begin{align}
\alpha(3,3,3)&=\alpha(3,0,3)=\alpha(0,3,3)=\alpha(3,3,0)=I_{\left(\begin{smallmatrix}1&1&1\\1&1&1\\1&1&1\end{smallmatrix}\right)}=\nonumber\\
&=r^{n-1}r^{n-1}r^{n-2}-r^{n-1}r^{n-2}r^{n-2}-r^{n-1}r^{n-2}r^{n-2}-\nonumber\\&-r^{n-1}r^{n-2}r^{n-2}
+2r^{n-2}r^{n-2}r^{n-2}=r^{3n-6}(r^2-3r+2).\nonumber
\end{align}
We observe by the last remark, we have
\begin{align}
\alpha(3,2,3)&=\alpha(2,3,3)=\alpha(3,3,2)=\alpha(3,1,3)=\alpha(1,3,3)=\alpha(3,3,1)=\nonumber\\
&=I_{\left(\begin{smallmatrix}1&1&1\\1&1&1\\1&0&0\end{smallmatrix}\right)}+I_{\left(\begin{smallmatrix}1&1&1\\1&1&1\\0&1&0\end{smallmatrix}\right)}+I_{\left(\begin{smallmatrix}1&1&1\\1&1&1\\0&0&1\end{smallmatrix}\right)}=\nonumber\\
&=3r^{3n-6}(r^2-3r+2)\nonumber
\end{align}
and
\begin{align}
&\alpha(3,2,2)=\alpha(2,3,2)=\alpha(2,2,3)=I_{\left(\begin{smallmatrix}1&1&1\\1&1&0\\1&1&0\end{smallmatrix}\right)}+I_{\left(\begin{smallmatrix}1&1&1\\1&1&0\\1&0&1\end{smallmatrix}\right)}+I_{\left(\begin{smallmatrix}1&1&1\\1&1&0\\0&1&1\end{smallmatrix}\right)}+\nonumber\\
&+I_{\left(\begin{smallmatrix}1&1&1\\1&0&1\\1&1&0\end{smallmatrix}\right)}+I_{\left(\begin{smallmatrix}1&1&1\\1&0&1\\1&0&1\end{smallmatrix}\right)}+I_{\left(\begin{smallmatrix}1&1&1\\1&0&1\\0&1&1\end{smallmatrix}\right)}+I_{\left(\begin{smallmatrix}1&1&1\\0&1&1\\1&1&0\end{smallmatrix}\right)}+I_{\left(\begin{smallmatrix}1&1&1\\0&1&1\\1&0&1\end{smallmatrix}\right)}+I_{\left(\begin{smallmatrix}1&1&1\\0&1&1\\0&1&1\end{smallmatrix}\right)}=\nonumber\\
&=r^{3n-5}(-2r+1)+r^{3n-6}(r^2-3r+2)+r^{3n-6}(r^2-3r+2)+\nonumber\\&+r^{3n-6}(r^2-3r+2)+r^{3n-4}(r-1)+r^{3n-6}(r^2-3r+2)+\nonumber\\&+r^{3n-6}(r^2-3r+2)+r^{3n-6}(r^2-3r+2)+r^{3n-5}(r^2-3r+2)=\nonumber\\&=r^{3n-6}(2r^{3}-15r+12).\nonumber
\end{align}
We also obtain
\begin{align}
&\alpha(3,2,1)=\alpha(2,3,1)=\alpha(3,1,2)=\alpha(1,3,2)=\alpha(1,2,3)=\alpha(2,1,3)=\nonumber\\&=I_{\left(\begin{smallmatrix}1&1&1\\1&1&0\\1&0&0\end{smallmatrix}\right)}+I_{\left(\begin{smallmatrix}1&1&1\\1&1&0\\0&1&0\end{smallmatrix}\right)}+I_{\left(\begin{smallmatrix}1&1&1\\1&1&0\\0&0&1\end{smallmatrix}\right)}+\nonumber\\
&+I_{\left(\begin{smallmatrix}1&1&1\\1&0&1\\1&0&0\end{smallmatrix}\right)}+I_{\left(\begin{smallmatrix}1&1&1\\1&0&1\\0&1&0\end{smallmatrix}\right)}+I_{\left(\begin{smallmatrix}1&1&1\\1&0&1\\0&0&1\end{smallmatrix}\right)}+I_{\left(\begin{smallmatrix}1&1&1\\0&1&1\\1&0&0\end{smallmatrix}\right)}+I_{\left(\begin{smallmatrix}1&1&1\\0&1&1\\0&1&0\end{smallmatrix}\right)}+I_{\left(\begin{smallmatrix}1&1&1\\0&1&1\\0&0&1\end{smallmatrix}\right)}=\nonumber\\
&=r^{3n-5}(-2r+1)+r^{3n-5}(-2r+1)+r^{3n-6}(r^2-3r+2)+r^{3n-4}(r-1)+\nonumber\\&+r^{3n-6}(r^2-3r+2)+r^{3n-4}(r-1)+r^{3n-6}(r^2-3r+2)+\nonumber\\
&+r^{3n-5}(r^2-3r+2)+r^{3n-5}(r^2-3r+2)=r^{3n-6}(4r^{3}-9r^2-3r+6).\nonumber
\end{align}
For the element
\begin{align}
&\alpha(3,2,0)=\alpha(2,3,0)=\alpha(3,0,2)=\alpha(0,3,2)=\alpha(0,2,3)=\alpha(2,0,3)=\nonumber\\&=I_{\left(\begin{smallmatrix}1&1&1\\1&1&0\\0&0&0\end{smallmatrix}\right)}+I_{\left(\begin{smallmatrix}1&1&1\\1&0&1\\0&0&0\end{smallmatrix}\right)}+I_{\left(\begin{smallmatrix}1&1&1\\0&1&1\\0&0&0\end{smallmatrix}\right)}=\nonumber\\
&=r^{3n-5}(-2r+1)+r^{3n-4}(r-1)+r^{3n-5}(r^2-3r+2)=r^{3n-5}(2r^2-6r+3).\nonumber
\end{align}
Now, the element
\begin{align}
&\alpha(3,1,1)=\alpha(1,3,1)=\alpha(1,1,3)=I_{\left(\begin{smallmatrix}1&1&1\\1&0&0\\1&0&0\end{smallmatrix}\right)}+I_{\left(\begin{smallmatrix}1&1&1\\1&0&0\\0&1&0\end{smallmatrix}\right)}+I_{\left(\begin{smallmatrix}1&1&1\\1&0&0\\0&0&1\end{smallmatrix}\right)}+\nonumber\\
&+I_{\left(\begin{smallmatrix}1&1&1\\0&1&0\\1&0&0\end{smallmatrix}\right)}+I_{\left(\begin{smallmatrix}1&1&1\\0&1&0\\0&1&0\end{smallmatrix}\right)}+I_{\left(\begin{smallmatrix}1&1&1\\0&1&0\\0&0&1\end{smallmatrix}\right)}+I_{\left(\begin{smallmatrix}1&1&1\\0&0&1\\1&0&0\end{smallmatrix}\right)}+I_{\left(\begin{smallmatrix}1&1&1\\0&0&1\\0&1&0\end{smallmatrix}\right)}+I_{\left(\begin{smallmatrix}1&1&1\\0&0&1\\0&0&1\end{smallmatrix}\right)}=\nonumber\\
&=0+r^{3n-5}(-2r+1)+r^{3n-4}(r-1)+r^{3n-5}(-2r+1)+r^{3n-4}(-2r+2)+\nonumber\\&+r^{3n-5}(r^2-3r+2)+r^{3n-4}(r-1)+r^{3n-5}(r^2-3r+2)+r^{3n-3}(r-1)=\nonumber\\
&=r^{3n-5}(r^{3}+r^2-10r+6)\nonumber
\end{align}
and
\begin{align}
&\alpha(3,1,0)=\alpha(0,3,1)=\alpha(0,1,3)=\alpha(3,0,1)=\alpha(1,0,3)=\alpha(1,3,0)=\nonumber\\&=I_{\left(\begin{smallmatrix}1&1&1\\1&0&0\\0&0&0\end{smallmatrix}\right)}+I_{\left(\begin{smallmatrix}1&1&1\\0&1&0\\0&0&0\end{smallmatrix}\right)}+I_{\left(\begin{smallmatrix}1&1&1\\0&0&1\\0&0&0\end{smallmatrix}\right)}+\nonumber\\
&=0+r^{3n-4}(-2r+2)+r^{3n-3}(r-1)=\nonumber\\
&=r^{3n-4}(r^2-3r+2)\nonumber
\end{align}
also
\begin{align}
&\alpha(3,0,0)=\alpha(0,3,0)=\alpha(0,0,3)=I_{\left(\begin{smallmatrix}1&1&1\\0&0&0\\0&0&0\end{smallmatrix}\right)}=0.\nonumber
\end{align}
One of the longest terms is
\begin{align}
&\alpha(2,2,2)=I_{\left(\begin{smallmatrix}1&1&0\\1&1&0\\1&1&0\end{smallmatrix}\right)}+I_{\left(\begin{smallmatrix}1&1&0\\1&1&0\\1&0&1\end{smallmatrix}\right)}+I_{\left(\begin{smallmatrix}1&1&0\\1&1&0\\0&1&1\end{smallmatrix}\right)}+\nonumber\\
&+I_{\left(\begin{smallmatrix}1&1&0\\1&0&1\\1&1&0\end{smallmatrix}\right)}+I_{\left(\begin{smallmatrix}1&1&0\\1&0&1\\1&0&1\end{smallmatrix}\right)}+I_{\left(\begin{smallmatrix}1&1&0\\1&0&1\\0&1&1\end{smallmatrix}\right)}+I_{\left(\begin{smallmatrix}1&1&0\\0&1&1\\1&1&0\end{smallmatrix}\right)}+I_{\left(\begin{smallmatrix}1&1&0\\0&1&1\\1&0&1\end{smallmatrix}\right)}+I_{\left(\begin{smallmatrix}1&1&0\\0&1&1\\0&1&1\end{smallmatrix}\right)}+\nonumber\\
&+I_{\left(\begin{smallmatrix}1&0&1\\1&1&0\\1&1&0\end{smallmatrix}\right)}+I_{\left(\begin{smallmatrix}1&0&1\\1&1&0\\1&0&1\end{smallmatrix}\right)}+I_{\left(\begin{smallmatrix}1&0&1\\1&1&0\\0&1&1\end{smallmatrix}\right)}
+I_{\left(\begin{smallmatrix}1&0&1\\1&0&1\\1&1&0\end{smallmatrix}\right)}+I_{\left(\begin{smallmatrix}1&0&1\\1&0&1\\1&0&1\end{smallmatrix}\right)}+I_{\left(\begin{smallmatrix}1&0&1\\1&0&1\\0&1&1\end{smallmatrix}\right)}+\nonumber\\&+I_{\left(\begin{smallmatrix}1&0&1\\0&1&1\\1&1&0\end{smallmatrix}\right)}+I_{\left(\begin{smallmatrix}1&0&1\\0&1&1\\1&0&1\end{smallmatrix}\right)}+I_{\left(\begin{smallmatrix}1&0&1\\0&1&1\\0&1&1\end{smallmatrix}\right)}
+I_{\left(\begin{smallmatrix}0&1&1\\1&1&0\\1&1&0\end{smallmatrix}\right)}+I_{\left(\begin{smallmatrix}0&1&1\\1&1&0\\1&0&1\end{smallmatrix}\right)}+I_{\left(\begin{smallmatrix}0&1&1\\1&1&0\\0&1&1\end{smallmatrix}\right)}+\nonumber\\
&+I_{\left(\begin{smallmatrix}0&1&1\\1&0&1\\1&1&0\end{smallmatrix}\right)}+I_{\left(\begin{smallmatrix}0&1&1\\1&0&1\\1&0&1\end{smallmatrix}\right)}+I_{\left(\begin{smallmatrix}0&1&1\\1&0&1\\0&1&1\end{smallmatrix}\right)}+I_{\left(\begin{smallmatrix}0&1&1\\0&1&1\\1&1&0\end{smallmatrix}\right)}+I_{\left(\begin{smallmatrix}0&1&1\\0&1&1\\1&0&1\end{smallmatrix}\right)}+I_{\left(\begin{smallmatrix}0&1&1\\0&1&1\\0&1&1\end{smallmatrix}\right)}=\nonumber
\end{align}
\begin{align}
\hspace{1cm}
&=r^{3n-4}(-3r+3)+r^{3n-5}(-2r+2)+r^{3n-5}(-2r+2)+r^{3n-5}(-2r+2)+\nonumber\\
&+r^{3n-4}(r-1)+r^{3n-6}(r^2-3r+2)+r^{3n-5}(-2r+2)+r^{3n-6}(r^2-3r+2)+\nonumber\\
&+r^{3n-5}(r^2-3r+2)+r^{3n-5}(-2r+2)+r^{3n-4}(r-1)+r^{3n-6}(r^2-3r+2)+\nonumber\\
&+r^{3n-4}(r-1)+r^{3n-5}(r^3-3r+2)+r^{3n-4}(r-1)+r^{3n-6}(r^2-3r+2)+\nonumber\\
&+r^{3n-4}(r-1)+r^{3n-5}(r^2-3r+2)+r^{3n-5}(-2r+2)+\nonumber\\
&+r^{3n-6}(r^2-3r+2)+r^{3n-5}(r^2-3r+2)+r^{3n-6}(r^2-3r+2)+\nonumber\\
&+r^{3n-4}(r-1)+r^{3n-5}(r^2-3r+2)+r^{3n-5}(r^2-3r+2)+\nonumber\\
&+r^{3n-5}(r^2-3r+2)+r^{3n-4}(r^2-3r+2)=\nonumber\\
&=r^{3n-6}(2r^4+6r^3-28r^2+8r+12)\nonumber
\end{align}
and
\begin{align}
&\alpha(2,2,1)=\alpha(2,1,2)=\alpha(1,2,2)=I_{\left(\begin{smallmatrix}1&1&0\\1&1&0\\1&0&0\end{smallmatrix}\right)}+I_{\left(\begin{smallmatrix}1&1&0\\1&1&0\\0&1&0\end{smallmatrix}\right)}+I_{\left(\begin{smallmatrix}1&1&0\\1&1&0\\0&0&1\end{smallmatrix}\right)}+\nonumber\\
&+I_{\left(\begin{smallmatrix}1&1&0\\1&0&1\\1&0&0\end{smallmatrix}\right)}+I_{\left(\begin{smallmatrix}1&1&0\\1&0&1\\0&1&0\end{smallmatrix}\right)}+I_{\left(\begin{smallmatrix}1&1&0\\1&0&1\\0&0&1\end{smallmatrix}\right)}+I_{\left(\begin{smallmatrix}1&1&0\\0&1&1\\1&0&0\end{smallmatrix}\right)}+I_{\left(\begin{smallmatrix}1&1&0\\0&1&1\\0&1&0\end{smallmatrix}\right)}+I_{\left(\begin{smallmatrix}1&1&0\\0&1&1\\0&0&1\end{smallmatrix}\right)}+\nonumber\\
&+I_{\left(\begin{smallmatrix}1&0&1\\1&1&0\\1&0&0\end{smallmatrix}\right)}+I_{\left(\begin{smallmatrix}1&0&1\\1&1&0\\0&1&0\end{smallmatrix}\right)}+I_{\left(\begin{smallmatrix}1&0&1\\1&1&0\\0&0&1\end{smallmatrix}\right)}
+I_{\left(\begin{smallmatrix}1&0&1\\1&0&1\\1&0&0\end{smallmatrix}\right)}+I_{\left(\begin{smallmatrix}1&0&1\\1&0&1\\0&1&0\end{smallmatrix}\right)}+I_{\left(\begin{smallmatrix}1&0&1\\1&0&1\\0&0&1\end{smallmatrix}\right)}+\nonumber\\&+I_{\left(\begin{smallmatrix}1&0&1\\0&1&1\\1&0&0\end{smallmatrix}\right)}+I_{\left(\begin{smallmatrix}1&0&1\\0&1&1\\0&1&0\end{smallmatrix}\right)}+I_{\left(\begin{smallmatrix}1&0&1\\0&1&1\\0&0&1\end{smallmatrix}\right)}
+I_{\left(\begin{smallmatrix}0&1&1\\1&1&0\\1&0&0\end{smallmatrix}\right)}+I_{\left(\begin{smallmatrix}0&1&1\\1&1&0\\0&1&0\end{smallmatrix}\right)}+I_{\left(\begin{smallmatrix}0&1&1\\1&1&0\\0&0&1\end{smallmatrix}\right)}+\nonumber\\
&+I_{\left(\begin{smallmatrix}0&1&1\\1&0&1\\1&0&0\end{smallmatrix}\right)}+I_{\left(\begin{smallmatrix}0&1&1\\1&0&1\\0&1&0\end{smallmatrix}\right)}+I_{\left(\begin{smallmatrix}0&1&1\\1&0&1\\0&0&1\end{smallmatrix}\right)}+I_{\left(\begin{smallmatrix}0&1&1\\0&1&1\\1&0&0\end{smallmatrix}\right)}+I_{\left(\begin{smallmatrix}0&1&1\\0&1&1\\0&1&0\end{smallmatrix}\right)}+I_{\left(\begin{smallmatrix}0&1&1\\0&1&1\\0&0&1\end{smallmatrix}\right)}=\nonumber
\end{align}
\begin{align}
&=r^{3n-4}(-3r+3)+r^{3n-4}(-3r+3)+r^{3n-5}(-2r+2)+0+\nonumber\\
&+r^{3n-5}(-2r+2)+r^{3n-4}(r-1)+r^{3n-5}(-2r+2)+r^{3n-4}(-2r+2)+\nonumber\\
&+r^{3n-5}(r^2-3r+2)+0+r^{3n-5}(-2r+2)+r^{3n-4}(r-1)+\nonumber\\
&+r^{3n-5}(r^3-3r+2)+r^{3n-4}(r-1)+r^{3n-5}(r^3-3r+2)+r^{3n-4}(r-1)+\nonumber\\
&+r^{3n-5}(r^2-3r+2)+r^{3n-3}(r-1)+r^{3n-5}(-2r+2)+r^{3n-4}(-2r+2)+\nonumber
\end{align}
\begin{align}
\hspace{1cm}
&+r^{3n-5}(r^2-3r+2)+r^{3n-4}(r-1)+r^{3n-5}(r^2-3r+2)+r^{3n-3}(r-1)+\nonumber\\
&+r^{3n-5}(r^2-3r+2)+r^{3n-4}(r^2-3r+2)+r^{3n-4}(r^2-3r+2)=\nonumber\\
&=r^{3n-5}(5r^3-7r^2-22r+24).\nonumber
\end{align}
Now, the element
\begin{align}
&\alpha(2,2,0)=\alpha(2,0,2)=\alpha(0,2,2)=I_{\left(\begin{smallmatrix}1&1&0\\1&1&0\\0&0&0\end{smallmatrix}\right)}+I_{\left(\begin{smallmatrix}1&1&0\\1&0&1\\0&0&0\end{smallmatrix}\right)}
+I_{\left(\begin{smallmatrix}1&1&0\\0&1&1\\0&0&0\end{smallmatrix}\right)}+\nonumber\\
&+I_{\left(\begin{smallmatrix}1&0&1\\1&1&0\\0&0&0\end{smallmatrix}\right)}+I_{\left(\begin{smallmatrix}1&0&1\\1&0&1\\0&0&0\end{smallmatrix}\right)}+
I_{\left(\begin{smallmatrix}1&0&1\\0&1&1\\0&0&0\end{smallmatrix}\right)}+I_{\left(\begin{smallmatrix}0&1&1\\1&1&0\\0&0&0\end{smallmatrix}\right)}+I_{\left(\begin{smallmatrix}0&1&1\\1&0&1\\0&0&0\end{smallmatrix}\right)}+I_{\left(\begin{smallmatrix}0&1&1\\0&1&1\\0&0&0\end{smallmatrix}\right)}\nonumber\\
&=r^{3n-4}(-3r+3)+0+r^{3n-4}(-2r+2)+0+r^{3n-5}(r^3-3r+2)\nonumber\\
&+r^{3n-3}(r-1)+r^{3n-4}(-2r+2)+r^{3n-3}(r-1)+r^{3n-4}(r^2-3r+2)\nonumber\\
&=r^{3n-5}(4r^3-12r^2+6r+2)\nonumber
\end{align}
and
\begin{align}
&\alpha(2,1,1)=\alpha(1,1,2)=\alpha(1,2,1)=I_{\left(\begin{smallmatrix}1&1&0\\1&0&0\\1&0&0\end{smallmatrix}\right)}+I_{\left(\begin{smallmatrix}1&1&0\\1&0&0\\0&1&0\end{smallmatrix}\right)}+I_{\left(\begin{smallmatrix}1&1&0\\1&0&0\\0&0&1\end{smallmatrix}\right)}+\nonumber\\
&+I_{\left(\begin{smallmatrix}1&1&0\\0&1&0\\1&0&0\end{smallmatrix}\right)}+I_{\left(\begin{smallmatrix}1&1&0\\0&1&0\\0&1&0\end{smallmatrix}\right)}+I_{\left(\begin{smallmatrix}1&1&0\\0&1&0\\0&0&1\end{smallmatrix}\right)}+I_{\left(\begin{smallmatrix}1&1&0\\0&0&1\\1&0&0\end{smallmatrix}\right)}+I_{\left(\begin{smallmatrix}1&1&0\\0&0&1\\0&1&0\end{smallmatrix}\right)}+I_{\left(\begin{smallmatrix}1&1&0\\0&0&1\\0&0&1\end{smallmatrix}\right)}+\nonumber\\
&+I_{\left(\begin{smallmatrix}1&0&1\\1&0&0\\1&0&0\end{smallmatrix}\right)}+I_{\left(\begin{smallmatrix}1&0&1\\1&0&0\\0&1&0\end{smallmatrix}\right)}+I_{\left(\begin{smallmatrix}1&0&1\\1&0&0\\0&0&1\end{smallmatrix}\right)}
+I_{\left(\begin{smallmatrix}1&0&1\\0&1&0\\1&0&0\end{smallmatrix}\right)}+I_{\left(\begin{smallmatrix}1&0&1\\0&1&0\\0&1&0\end{smallmatrix}\right)}+I_{\left(\begin{smallmatrix}1&0&1\\0&1&0\\0&0&1\end{smallmatrix}\right)}+\nonumber\\
&+I_{\left(\begin{smallmatrix}1&0&1\\0&0&1\\1&0&0\end{smallmatrix}\right)}+I_{\left(\begin{smallmatrix}1&0&1\\0&0&1\\0&1&0\end{smallmatrix}\right)}+I_{\left(\begin{smallmatrix}1&0&1\\0&0&1\\0&0&1\end{smallmatrix}\right)}
+I_{\left(\begin{smallmatrix}0&1&1\\1&0&0\\1&0&0\end{smallmatrix}\right)}+I_{\left(\begin{smallmatrix}0&1&1\\1&0&0\\0&1&0\end{smallmatrix}\right)}+I_{\left(\begin{smallmatrix}0&1&1\\1&0&0\\0&0&1\end{smallmatrix}\right)}+\nonumber\\
&+I_{\left(\begin{smallmatrix}0&1&1\\0&1&0\\1&0&0\end{smallmatrix}\right)}+I_{\left(\begin{smallmatrix}0&1&1\\0&1&0\\0&1&0\end{smallmatrix}\right)}+I_{\left(\begin{smallmatrix}0&1&1\\0&1&0\\0&0&1\end{smallmatrix}\right)}+I_{\left(\begin{smallmatrix}0&1&1\\0&0&1\\1&0&0\end{smallmatrix}\right)}+I_{\left(\begin{smallmatrix}0&1&1\\0&0&1\\0&1&0\end{smallmatrix}\right)}+I_{\left(\begin{smallmatrix}0&1&1\\0&0&1\\0&0&1\end{smallmatrix}\right)}\nonumber\\
&=r^{3n-3}(-r+1)+r^{3n-4}(-3r+3)+0+r^{3n-4}(-3r+3)+r^{3n-3}(-3r+3)+\nonumber\\
&+r^{3n-4}(-2r+2)+0+r^{3n-4}(-2r+2)+r^{3n-3}(r-1)+r^{3n-3}(r-1)+\nonumber\\
&+0+r^{3n-5}(r^3-3r+2)+0+r^{3n-4}(-2r+2)+r^{3n-3}(r-1)+\nonumber\\
&+r^{3n-5}(r^3-3r+2)+r^{3n-3}(r-1)+r^{3n-3}(r^2-1)+0+r^{3n-4}(-2r+2)\nonumber\\
&+r^{3n-3}(r-1)+r^{3n-4}(-2r+2)+r^{3n-3}(-2r+2)+r^{3n-4}(r^2-3r+2)+\nonumber\\
&+r^{3n-3}(r-1)
+r^{3n-4}(r^2-3r+2)+r^{3n-2}(r-1)=\nonumber\\&=r^{3n-5}(2r^4+3r^3-23r^2+14r+4).\nonumber
\end{align}
We also obtain
\begin{align}
&\alpha(2,1,0)=\alpha(0,1,2)=\alpha(0,2,1)=\alpha(2,0,1)=\alpha(1,0,2)=\alpha(1,2,0)=\nonumber\\&=I_{\left(\begin{smallmatrix}1&1&0\\1&0&0\\0&0&0\end{smallmatrix}\right)}+
+I_{\left(\begin{smallmatrix}1&1&0\\0&1&0\\0&0&0\end{smallmatrix}\right)}+I_{\left(\begin{smallmatrix}1&1&0\\0&0&1\\0&0&0\end{smallmatrix}\right)}+
+I_{\left(\begin{smallmatrix}1&0&1\\1&0&0\\0&0&0\end{smallmatrix}\right)}+I_{\left(\begin{smallmatrix}1&0&1\\0&1&0\\0&0&0\end{smallmatrix}\right)}+
+I_{\left(\begin{smallmatrix}1&0&1\\0&0&1\\0&0&0\end{smallmatrix}\right)}+\nonumber\\&+I_{\left(\begin{smallmatrix}0&1&1\\1&0&0\\0&0&0\end{smallmatrix}\right)}+I_{\left(\begin{smallmatrix}0&1&1\\0&1&0\\0&0&0\end{smallmatrix}\right)}
+I_{\left(\begin{smallmatrix}0&1&1\\0&0&1\\0&0&0\end{smallmatrix}\right)}=
=r^{3n-3}(-r+1)+r^{3n-3}(-3r+3)+0+\nonumber\\
&+r^{3n-3}(r-1)+0+r^{3n-3}(r^2-1)+0+r^{3n-3}(-2r+2)+r^{3n-2}(r-1)=\nonumber\\
&=r^{3n-3}(2r^2-6r+4).\nonumber
\end{align}
Here, we get
\begin{align}
&\alpha(2,0,0)=\alpha(0,0,2)=\alpha(2,0,0)=I_{\left(\begin{smallmatrix}1&1&0\\0&0&0\\0&0&0\end{smallmatrix}\right)}+
I_{\left(\begin{smallmatrix}1&0&1\\0&0&0\\0&0&0\end{smallmatrix}\right)}+I_{\left(\begin{smallmatrix}0&1&1\\0&0&0\\0&0&0\end{smallmatrix}\right)}=\nonumber\\
&=r^{3n-2}(-r+1)+r^{3n-2}(r-1)+0=0\nonumber
\end{align}
and we can also get

\begin{align}
&\alpha(1,1,1)=I_{\left(\begin{smallmatrix}1&0&0\\1&0&0\\1&0&0\end{smallmatrix}\right)}+I_{\left(\begin{smallmatrix}1&0&0\\1&0&0\\0&1&0\end{smallmatrix}\right)}+I_{\left(\begin{smallmatrix}1&0&0\\1&0&0\\0&0&1\end{smallmatrix}\right)}+\nonumber\\
&+I_{\left(\begin{smallmatrix}1&0&0\\0&1&0\\1&0&0\end{smallmatrix}\right)}+I_{\left(\begin{smallmatrix}1&0&0\\0&1&0\\0&1&0\end{smallmatrix}\right)}+I_{\left(\begin{smallmatrix}1&0&0\\0&1&0\\0&0&1\end{smallmatrix}\right)}+I_{\left(\begin{smallmatrix}1&0&0\\0&0&1\\1&0&0\end{smallmatrix}\right)}+I_{\left(\begin{smallmatrix}1&0&0\\0&0&1\\0&1&0\end{smallmatrix}\right)}+I_{\left(\begin{smallmatrix}1&0&0\\0&0&1\\0&0&1\end{smallmatrix}\right)}+\nonumber\\
&+I_{\left(\begin{smallmatrix}0&1&0\\1&0&0\\1&0&0\end{smallmatrix}\right)}+I_{\left(\begin{smallmatrix}0&1&0\\1&0&0\\0&1&0\end{smallmatrix}\right)}+I_{\left(\begin{smallmatrix}0&1&0\\1&0&0\\0&0&1\end{smallmatrix}\right)}
+I_{\left(\begin{smallmatrix}0&1&0\\0&1&0\\1&0&0\end{smallmatrix}\right)}+I_{\left(\begin{smallmatrix}0&1&0\\0&1&0\\0&1&0\end{smallmatrix}\right)}+I_{\left(\begin{smallmatrix}0&1&0\\0&1&0\\0&0&1\end{smallmatrix}\right)}+\nonumber\\
&+I_{\left(\begin{smallmatrix}0&1&0\\0&0&1\\1&0&0\end{smallmatrix}\right)}+I_{\left(\begin{smallmatrix}0&1&0\\0&0&1\\0&1&0\end{smallmatrix}\right)}+I_{\left(\begin{smallmatrix}0&1&0\\0&0&1\\0&0&1\end{smallmatrix}\right)}
+I_{\left(\begin{smallmatrix}0&0&1\\1&0&0\\1&0&0\end{smallmatrix}\right)}+I_{\left(\begin{smallmatrix}0&0&1\\1&0&0\\0&1&0\end{smallmatrix}\right)}+I_{\left(\begin{smallmatrix}0&0&1\\1&0&0\\0&0&1\end{smallmatrix}\right)}+\nonumber\\
&+I_{\left(\begin{smallmatrix}0&0&1\\0&1&0\\1&0&0\end{smallmatrix}\right)}+I_{\left(\begin{smallmatrix}0&0&1\\0&1&0\\0&1&0\end{smallmatrix}\right)}+I_{\left(\begin{smallmatrix}0&0&1\\0&1&0\\0&0&1\end{smallmatrix}\right)}+I_{\left(\begin{smallmatrix}0&0&1\\0&0&1\\1&0&0\end{smallmatrix}\right)}+I_{\left(\begin{smallmatrix}0&0&1\\0&0&1\\0&1&0\end{smallmatrix}\right)}+I_{\left(\begin{smallmatrix}0&0&1\\0&0&1\\0&0&1\end{smallmatrix}\right)}=\nonumber\\
&=0+r^{3n-3}(-r+1)+r^{3n-3}(r-1)+r^{3n-3}(-r+1)+r^{3n-3}(-3r+3)\nonumber\\
&+0+r^{3n-3}(r-1)+0+r^{3n-3}(r^2-1)+r^{3n-3}(-r+1)+\nonumber\\
&+r^{3n-3}(-3r+3)+0+r^{3n-3}(-3r+3)+r^{3n-2}(-3r+3)+r^{3n-3}(-2r+2)\nonumber\\
&+0+r^{3n-3}(-2r+2)+r^{3n-2}(r-1)+r^{3n-3}(r-1)+0+\nonumber\\
&+r^{3n-3}(r^2-1)+0+r^{3n-3}(-2r+2)+r^{3n-2}(r-1)+r^{3n-3}(r^2-1)\nonumber\\
&+r^{3n-2}(r-1)+r^{3n-2}(r^2-1)=r^{3n-3}(r^3+3r^2-16r+12).\nonumber
\end{align}
For the element
\begin{align}
&\alpha(1,1,0)=\alpha(1,0,1)=\alpha(0,1,1)=I_{\left(\begin{smallmatrix}1&0&0\\1&0&0\\0&0&0\end{smallmatrix}\right)}+I_{\left(\begin{smallmatrix}1&0&0\\0&1&0\\0&0&0\end{smallmatrix}\right)}+I_{\left(\begin{smallmatrix}1&0&0\\0&0&1\\0&0&0\end{smallmatrix}\right)}+\nonumber\\&
+I_{\left(\begin{smallmatrix}0&1&0\\1&0&0\\0&0&0\end{smallmatrix}\right)}+I_{\left(\begin{smallmatrix}0&1&0\\0&1&0\\0&0&0\end{smallmatrix}\right)}+I_{\left(\begin{smallmatrix}0&1&0\\0&0&1\\0&0&0\end{smallmatrix}\right)}+I_{\left(\begin{smallmatrix}0&0&1\\1&0&0\\0&0&0\end{smallmatrix}\right)}+I_{\left(\begin{smallmatrix}0&0&1\\0&1&0\\0&0&0\end{smallmatrix}\right)}
+I_{\left(\begin{smallmatrix}0&0&1\\0&0&1\\0&0&0\end{smallmatrix}\right)}=\nonumber\\
&=0+r^{3n-2}(-r+1)+r^{3n-2}(r-1)+r^{3n-2}(-r+1)+r^{3n-2}(-3r+3)\nonumber\\
&=0+r^{3n-2}(r-1)+0+r^{3n-2}(r^2-1)=\nonumber\\
&=r^{3n-2}(r^2-3r+2)\nonumber
\end{align}
and
\begin{align}
&\alpha(1,0,0)=\alpha(0,0,1)=\alpha(0,1,0)=I_{\left(\begin{smallmatrix}1&0&0\\0&0&0\\0&0&0\end{smallmatrix}\right)}+I_{\left(\begin{smallmatrix}0&1&0\\0&0&0\\0&0&0\end{smallmatrix}\right)}+I_{\left(\begin{smallmatrix}0&0&1\\0&0&0\\0&0&0\end{smallmatrix}\right)}=\nonumber\\
&=0+r^{3n-2}(-r+1)+r^{3n-1}(-r+1)+r^{3n-1}(r-1)=0.\nonumber
\end{align}
Finally, we have
\begin{align}
&\alpha(0,0,0)=I_{\left(\begin{smallmatrix}0&0&0\\0&0&0\\0&0&0\end{smallmatrix}\right)}=0\nonumber
\end{align}
At this moment we can prove the theorem~\ref{T:convexity} :
\begin{proof}(Theorem GHS inequality for the Potts model)
The study of the signal of the quantity $\frac{\partial^2 m_i(\underline{J},\underline{B})}{\partial B_j\partial B_k}$ is related to the signal of $I(\underline{J},\underline{B})$ (see equation (\ref{E:def_I})). By the separation formula the signal of $I(\underline{J},\underline{B})$ is defined by the signal of the polynomial
\begin{align}
&\sum_{\begin{array}{c}A\in \mathcal{A}_{\binom{N+1}{2}}\\\mbox{if }o_p\notin O_2\\a(p,1)=a(p,2)=a(p,3)=0\end{array}}I_AX_{p_1}^{n(p_1)}X_{p_2}^{n(p_2)}X_{p_3}^{n(p_3)}=\nonumber\\&=\sum_{n(p_1),n(p_2),n(p_3)}\alpha(n(p_1),n(p_2),n(p_3))X_{p_1}^{n(p_1)}X_{p_2}^{n(p_2)}X_{p_3}^{n(p_3)}
\end{align}
But we can easily verify that the coefficients $\alpha(n(p_1),n(p_2),n(p_3))$ are all of them non positive if $r=2$ and are all of them non-negative if $r\geq 3$.
\qed
\end{proof}

\end{document}